\documentclass{article}%
\usepackage{amsmath}
\usepackage{amsfonts}
\usepackage{amssymb}
\usepackage{graphicx}
\usepackage{caption}%
\providecommand{\U}[1]{\protect\rule{.1in}{.1in}}
\newtheorem{theorem}{Theorem}[section]

\newtheorem{corollary}[theorem]{Corollary}

\newtheorem{remark}[theorem]{Remark}

\newenvironment{proof}[1][Proof]{\noindent\textbf{#1.} }{\ \rule{0.5em}{0.5em}}
\begin{document}

\title{Revisiting the mathematical synthesis of the laws of Kepler and Galileo
leading to Newton's law of universal gravitation }
\author{Wu-Yi Hsiang\\Department of Mathematics\\University of California, Berkeley
\and Eldar Straume\\Department of Mathematical Sciences\\Norwegian University of Science and \\Technology, Trondheim, Norway}
\maketitle

\begin{abstract}
Newton's deduction of the inverse square law from Kepler's ellipse and area
laws together with his \textquotedblleft superb theorem\textquotedblright\ on
the gravitation attraction of spherically symmetric bodies, are the major
steps leading to the discovery of the law of universal gravitation (Principia,
1687). The goal of this article is to revisit some "well-known" events in the
history of science, and moreover, to provide elementary and clean-cut proofs,
still in the spirit of Newton, of these major advances. Being accessible to
first year university students, the educational aspect of such a coherent
presentation should not be overlooked.

\end{abstract}
\tableofcontents

\section{Introduction}

Towards the later half of the 16th century, the Renaissance of Greek
civilization in Europe had paved the way for major advancements leading to the
creation of modern science. In astronomy, the heliocentric theory of
Copernicus and the systematic astronomical observational data of Tycho de
Brahe led to the discovery of the magnificent empirical laws of Kepler on
planetary motions. In physics, the empirical laws of Galileo on terrestrial
gravity (i.e. free falling bodies and motions on inclined planes) laid the
modern foundation of mechanics.\ Then, in the later half of the 17th century,
the mathematical synthesis of those empirical laws led Isaac Newton
(1643--1727) to the discovery of the law of universal gravitation, published
in his monumental treatise \emph{Philosophiae Naturalis Principia
Mathematicae}\textit{\ (1687), }divided into three books, often referred to
as\textit{ }\emph{Principia }for brevity.

Note that, at this junction of scientific\ advancements, it is the
mathematical analysis as well as the synthesis of empirical laws that play the
unique role leading towards the underlying fundamental laws. Among them, the
first major step is the deduction of the inverse square law from Kepler's area
law and ellipse law, which was first achieved by Newton most likely sometimes
in 1680--81, and another most crucial step is the proof of the integration
formula for the gravitation attraction of spherically symmetric bodies,
nowadays often referred to as the \textquotedblleft superb
theorem\textquotedblright. In \emph{Principia}, these major results are
respectively stated as Proposition 11 and Proposition 71 of Book I, where
Newton deals with the laws of motion in vacuum.

Certainly, Newton's celebrated treatise is widely worshipped also today.
However, his proofs of the two most crucial propositions are quite difficult
to understand, even among university graduates. Therefore, a major goal of
this article is to provide elementary and clean-cut proofs of these results.

In the Prelude we shall briefly recall the major historical events and the
central figures contributing to the gradual understanding of planetary motions
prior to Newton. Moreover, the significance of the area law is analyzed in the
spirit of Newton's, but in modern terms of classical mechanics, for the
convenience of the reader.

Section 3 is devoted to a review of the focal and analytic geometry of the
ellipse, the type of curves playing such a predominant role in the work of
Kepler and Newton. Of particular interest is the curvature formula, which
Newton must have known, but for some reason appears only implicitly in his
proof of the inverse square law. In Section 4 we shall give three simple
different proofs of this law, allowing the modern reader to view Newton's
original and rather obscure proof with modern critical eyes. Moreover,
Newton's original and geometric proof of his \textquotedblleft superb
theorem\textquotedblright\ is not easy reading, which maybe explains the
occurrence of many new proofs in the more recent literature. We believe,
however, the proof given in Section 5 is the simplest one, in the spirit of
Newton's proof but based on one single new geometric idea.

Finally, for the sake of completeness, in Section 6 we shall solve what is
called the Kepler problem in the modern terminology, but was in fact referred
to as the \textquotedblleft inverse problem\textquotedblright\ at the time of
Newton and some time after, see \cite{Speis}. We present two proofs, one of
them is rather of standard type, reducing the integration to the solution of
Binet's equation, which in this case is a well known and simple second order
ODE in elementary calculus with $\sin(x)$ and $\cos(x)$ as solutions. But
there is an even simpler proof, where the integration problem is just finding
the antiderivatives of these two functions.

One may contend that "simple" proofs of the above classical problems are
nowadays found in a variety of calculus textbooks or elsewhere, so is there
anything new at all? The original motivation for the present article came,
however, from reading Chandrasekhar [3], which encouraged us to write a rather
short but coherent presentation accessible for the "common reader", in a
historical perspective and in the spirit of Newton's original approach.

\section{Prelude}

\subsection{Astronomy and geometry of the antiquity up to the Renaissance}

The ancient civilization of Egypt and Babylon had already accumulated a wealth
of astronomical and geometric knowledges that those great minds of Greek
civilization such as Thales (ca. 624--547 BC), Pythagoras (ca. 569--475 BC)
etc. gladly inherited, studied and deeply reflected upon. In particular,
Pythagoras pioneered the philosophical belief that the basic structures of the
universe are harmonious and based upon simple fundamental principles, while
the way to understand them is by studying numbers, ratios, and shapes. Ever
since then, his remarkable philosophical foresight still inspires generations
after generations of rational minds.

Following such a pioneering philosophy, the Pythagoreans devoted their studies
to geometry and astronomy and subsequently, geometry and astronomy became the
two major sciences of the antiquity, developing hand in hand. For example,
those great geometers of antiquity such as Eudoxus (408--355 BC), Archimedes
(287--212 BC), and Apollonius (ca. 262--190 BC), all had important
contributions to astronomy, while those great astronomers of antiquity such as
Aristarchus(ca. 310--230 BC), Hipparchus(190--120 BC), and Ptolemy(ca. 85--165
AD) all had excellent geometric expertise. \ We mention here three well-known
treatises that can be regarded as the embodiment of the glory of scientific
achievements of the antiquity, namely

\begin{itemize}
\item Euclid's \emph{Elements} (13 books)

\item Apollonius: \emph{Conics} (8 books)

\item Ptolemy: \emph{Almagest} (13 books)
\end{itemize}

\subsection{Copernicus, Tycho de Brahe, and Kepler: The new astronomy}

Today, it is a common knowledge that the Earth is just one of the planets
circulating around the sun. But this common knowledge was, in fact, the
monumental achievement of the new astronomy, culminating the successive
life-long devotions of Copernicus (1473--1543), Tycho de Brahe (1546--1601),
and Kepler (1571--1630). Kepler finally succeeded in solving the problem on
planetary motions (see below) that had been puzzling the civilization of
rational mind for many millenniums.

In the era of the Renaissance, Euclid's \emph{Elements} and Ptolemy's
\emph{Almagest} were used as important text books on geometry and astronomy at
major European universities. At Bologna University, Copernicus studied deeply
Ptolemy's \emph{Almagest }as an assistant of astronomy professor Navara, and
both of them were aware of \emph{Almagest}'s many shortcomings and troublesome
complexities. In 1514, inspired by the account of Archimedes on the
heliocentric theory of Aristarchus, he composed his decisive \emph{Commentary
}$(1515)$, outlining his own heliocentric theory which was finally completed
as the book \emph{De revolutionibus orbium coelestium }(1543)\emph{.
}Nowadays, this is commonly regarded as the heralding salvo of the modern
scientific revolution.

Note that a creditable astronomical theory must pass the test of accurate
predictions of verifiable astronomical events, such as observable events on
planetary motions. However, just a qualitatively sound heliocentric model of
the solar system would hardly be accepted as a well-established theory of
astronomy. Fortunately, almost like a divinely arranged \textquotedblleft
relay in astronomy\textquotedblright, the most diligent astronomical observer
Tycho de Brahe, with generous financial help from the King of Denmark and
Norway, made twenty years of superb astronomical observations at Uraniborg on
the island of Hven, and subsequently, Johannes Kepler became his assistant
(1600--1601), and moreover, succeeded him as Imperial Mathematician (of the
Holy Roman Empire, in Prague) after the sudden death of Tycho de Brahe. With
Kepler's superb mathematical expertise and marvelous creativity, it took him
20 years of hard work and devotion to finally succeed in solving the
millennium puzzle of planetary motions, namely the following remarkable
Kepler's laws, which we may state as follows:

\ \ \ \ \ \ 

\underline{Kepler's first law} (the ellipse law): The planets move on
elliptical orbits with the sun situated at one of the foci.

\underline{Kepler's second law} (the area law): The area per unit time
sweeping across by the line interval joining the planet to the sun is a
constant, as illustrated by Figure \ref{F1}.

\underline{Kepler's third law} (the period law): The ratio between the cube of
the major axis and the square of the period, namely $(2a)^{3}/T^{2}$, is the
same constant for all planets.

\ \ \ \ \ \ \ \ \ \ \ \ \ \ \ \ \ \ \
\begin{figure}[ptb]%
\centering
\includegraphics[
natheight=2.201000in,
natwidth=2.723800in,
height=2.2115in,
width=2.7308in
]%
{F:/KeplerNewtonHooke/Figures/Figure1.jpg}%
\caption{Illustration of the area law\label{F1}}%
\end{figure}

The following are the major publications of Kepler on his new astronomy:

\begin{itemize}
\item Astronomia Nova (1609)

\item 3 volumes of Epitome of Copernican astronomy (1618-1621)

\item Harmonice Mundi (1619)

\item Tabulae Rudolphinae (1627)
\end{itemize}

First of all, the predictions of the Rudolphine tables turned out to be
hundred times more accurate than that of the others. Moreover, Kepler
predicted the Mercury transit of Nov. 7, 1631 (which was observed in Paris by
P. Gassendi), and the Venus transit of Dec. 7, 1631 (that could not be
observed in Europe), while the next Venus transit would only occur after
another 130 years. Here, we would like to mention the remarkable achievements
of J. Horrocks (1618--1641). This brilliant young man was already fully in
command of Kepler's new astronomy at the age of 20, and after
\emph{correcting} Kepler's tables, he realized that a transit of Venus would
occur already on Nov. 24, 1639. His subsequent observation on the predicted
date, which he reported in \emph{Venus in Sole Visa}, was found in 1659, and
this is a noticeable triumph of the new astronomy.

\subsection{Galileo's empirical laws on terrestrial gravity, as evidence for
the inertia and force laws}

During his years at the University of Pisa (1589--92), Galileo Galilei
(1564--1642) wrote \emph{De Motu, }a series of essays on the theory of motion
(containing some mistakes, but was never published). Perhaps his most
important new idea in \emph{De Motu} is that one can test theories by
conducting experiments, such as testing his theory on falling bodies using an
inclined plane to vary the rate of descent.

In the years 1602--04 at Padua, he had returned to the dynamical study of
terrestrial gravity by conducting experiments on the inclined plane and the
pendulum. He had, by then, formulated the correct law of falling bodies and
worked out that a projectile follows a parabolic path. However, these
important results that laid the foundation of modern mechanics were only
published 35 years later in \emph{Discourses and mathematical demonstrations
concerning the two new sciences}. Here one finds the origin of the law of
inertia, in the sense that Galileo's conception of inertia is tantamount to
\emph{Newton's first law} of motion. Furthermore, Galileo's experiments on
falling bodies pointed toward the general force law (i.e. $\overrightarrow{F}%
=m\overrightarrow{a}$ ), namely \emph{Newton's second law}, which was
certainly also known to Newton's contemporaries Huygens, Halley, and Hooke.
However, the notion of "force" was, to some extent, already present in the
work of Archimedes on statics. On the other other hand, \emph{Newton's third
law} concerning mutually interacting forces, was a major innovation due to
himself, which we shall return to in Section 7.

\subsection{Equivalence between the area law and the action of a centripetal
force}

According to B. Cohen (cf. \cite{Cohen}, pp.167--169), a decisive step on the
path to universal gravity came in late 1679 and early 1680, when Robert Hooke
(1635-1703) introduced Newton to a new way of analyzing motion along a curved
trajectory, cf. Koyr\'{e} \cite{Koyre}. Hooke had cleverly seen that the
motion of an arbitrary body can be regarded as the combination of an inertia
component and a centripetal component. But he was unable to express this in a
more precise mathematical language. However, the possible influence of Hooke
on Newton's \emph{Principia} still engages many scientists and historians, cf.
e.g. Purrington \cite{Purrington}.

The terms \textquotedblleft centripetal force\textquotedblright\ and
\textquotedblleft radial force\textquotedblright\ will be used
interchangeably.The very first proposition of the \emph{Principia} develops
the dynamical significance of the law of areas by proving the mathematical
equivalence between the area law and the centripetality of the force, using
Hooke's technique. We include here a slight simplification of Newton's proof
in terms of modern terminology, namely

\begin{theorem}
\label{radial}Let $\overrightarrow{OP}$ be the position vector of a point mass
at $P$ moving in a plane, and let $\frac{dA}{dt}$ be the area swept out by
$\overrightarrow{OP}$ per unit time. Then $\frac{dA}{dt}$ is a constant if and
only if $\overrightarrow{OP}$ and the acceleration vector $\overrightarrow{a}$
are collinear, namely the force is centripetal.
\end{theorem}

\begin{proof}
The velocity vector $\overrightarrow{v}$ $=\frac{d}{dt}\overrightarrow{OP}$
and the acceleration vector $\overrightarrow{a}$ $=\frac{d^{2}}{dt^{2}%
}\overrightarrow{OP}$ lie in the plane of the motion, with unit normal vector
$\overrightarrow{n}$, say. Then we can write
\begin{equation}
2\frac{dA}{dt}=\overrightarrow{(OP}\times\overrightarrow{v})\cdot
\overrightarrow{n} \label{ang}%
\end{equation}
On the other hand
\begin{equation}
\frac{d}{dt}(\overrightarrow{OP}\times\overrightarrow{v})=\overrightarrow{v}%
\times\overrightarrow{v}+\overrightarrow{OP}\times\overrightarrow{a}%
=\overrightarrow{OP}\times\overrightarrow{a}, \label{ang1}%
\end{equation}
and by differentiating both sides of equation (\ref{ang}), we conclude that
$\frac{dA}{dt}$ is a constant if and only if $\overrightarrow{OP}%
\times\overrightarrow{a}=0$, which simply means $\overrightarrow{OP}$ and
$\overrightarrow{a}$ are collinear.
\end{proof}

\begin{remark}
By (\ref{ang1}),\ centripetality of the force acting on $P$ means the vector
$\overrightarrow{OP}\times\overrightarrow{v}$ is constant during the motion,
and clearly it is also normal to the motion. In particular, centripetality
implies the point moves in a plane. However, in the above theorem the meaning
of \textquotedblleft area swept out\textquotedblright\ needs no further
explanation since the motion is by assumption confined to a fixed plane.
\end{remark}

\begin{remark}
\label{area law}Let $(r,\theta)$ be polar coordinates centered at the point
$O$, hence $r=|\overrightarrow{OP}|$ and $\dot{\theta}=\frac{d\theta}{dt}$ is
the \emph{angular veclocity}. Then the quantity $\frac{dA}{dt}$ expresses as
\begin{equation}
2\frac{dA}{dt}=r^{2}\dot{\theta}=k \label{2}%
\end{equation}
In particular, for a planet whose trajectory is the ellipse with the sun at
the focal point $O$, it follows from Kepler's area law that the quantity
(\ref{2}) is the constant $\frac{2\pi ab}{T}$.
\end{remark}

\begin{remark}
(i) In Book 1 of \emph{Principia}, Newton paid much attention to centripetal
forces $F(r)$, asking two natural questions, namely (i) for a given trajectory
curve, what is the attracting force $F(r)$, and conversely, (ii) for a given
force law such as $F(r)\sim r^{n},n=1,-2,-3,-5$, what curves are the
corresponding trajectories ? For a few decades the problems were,
respectively, referred to as the \textquotedblleft direct\textquotedblright%
\ and \textquotedblleft inverse\textquotedblright\ (Kepler) problem (cf.
Speiser \cite{Speis}), which is rather peculiar since the terms
\textquotedblleft direct\textquotedblright\ and \textquotedblleft
inverse\textquotedblright\ later became switched, and this is also the modern terminology.

(ii) Making a leap forward to J.P.B. Binet (1786--1856), the \emph{Binet
equation}
\begin{equation}
F(q^{-1})=-mk^{2}q^{2}\left(  \frac{d^{2}q}{d\theta^{2}}+q\right)  \text{,
\ }q=1/r\text{, } \label{Binet}%
\end{equation}
is providing a unifying approach to the central force problem. We shall
illustrate its usage by applying it both to the inverse Kepler problem (in
Section 4) and the Kepler problem (in Section 6). For studies of the inverse
Kepler problem in the physics literature, see for example Ram\cite{Ram} and
Sivardi\`{e}re\cite{Siv}.
\end{remark}

In what follows, the notation $\overline{PQ}$ is used both for the segment
between points $P$ and $Q$ and the length of the segment.

\section{On the focal geometry of the ellipse}

In Greek geometry, the shape of ellipses first occurred as the tilted plane
sections of a circular cylinder, as indicated in Figure \ref{F2}, while
\textquotedblleft ellipse\textquotedblright\ means \textquotedblleft
non-circular\textquotedblright\ or distorted circle. However, the discovery of
its remarkable geometric characterization greatly excited the enthusiasm of
studying such a natural generalization of circular shapes, namely

\begin{theorem}
An ellipse $\Gamma$ has two foci $\{F_{1},F_{2}\}$ such that the sum of
$\ \overline{PF_{1}}$ and $\overline{PF_{2}}$ is equal to a constant for all
points $P$ on $\Gamma$.
\end{theorem}

\begin{proof}
Referring to Figure \ref{F2}, $Z$ is a circular cylinder cut by a plane $\Pi$
and $\Gamma=Z\cap\Pi$ is the plane section. Let $\Sigma_{1}$ (resp.
$\Sigma_{2})$ be the spheres of the same radius, inscribed and tangent to $Z$,
which are tangent to $\Pi$ at $F_{1}$( resp. $F_{2}$). Then, for any
$P\in\Gamma$, one has $\overline{PF_{i}}=\overline{PQ_{i}}$, and hence%
\begin{equation}
\overline{PF_{1}}+\overline{PF_{2}}=\overline{PQ_{1}}+\overline{PQ_{2}%
}=\overline{Q_{1}Q_{2}}=\text{constant} \label{sum}%
\end{equation}
%

\begin{figure}[ptb]%
\centering
\includegraphics[
natheight=5.680100in,
natwidth=3.494900in,
height=4.9783in,
width=3.0765in
]%
{F:/KeplerNewtonHooke/Figures/Figure2.jpg}%
\caption{Ancient geometric proof of (\ref{sum})\label{F2}}%
\end{figure}

\end{proof}

\subsection{The optical property of the ellipse}

\begin{theorem}
Let $P$ be a point on an ellipse $\Gamma$ with $\left\{  F_{1},F_{2}\right\}
$ as the pair of foci. Then, the tangent $\mathcal{T}_{P}$ (resp. normal
$\mathcal{V}_{P})$ bisects the outer (resp. inner) angle of $\Delta
F_{1}PF_{2}$ at $P$.
\end{theorem}

\begin{proof}
Let $l$ be the bisector of the outer angle of $\Delta F_{1}PF_{2}$ at $P$, and
$Q$ be another point on $l$. As indicated in Figure \ref{F3a} , $F_{2}%
^{\prime}$ is the reflection point of $F_{2}$ w.r.t. $l$. Then $\overline
{QF_{2}}=\overline{QF_{2}^{\prime}}$ and hence%
\[
\overline{QF_{1}}+\overline{QF_{2}}=\overline{QF_{1}}+\overline{QF_{2}%
^{\prime}}>\overline{F_{1}F_{2}^{\prime}}=\overline{PF_{1}}+\overline{PF_{2}%
}\text{,}%
\]
where the last identity follows by considering the angles at $P$, showing that
$P$ must, in fact, lie on the line through $F_{1}$ and $F_{2}^{\prime}$. Thus,
$Q$ must be outside of $\Gamma$, meaning that $l=$ $\mathcal{T}_{P}$ (i.e.
$l\cap\Gamma=\left\{  P\right\}  $). Now the statement about the normal
$\mathcal{V}_{P}$ follows immediately.
\end{proof}

\ \ \ \ \ \
\begin{figure}[ptb]%
\centering
\includegraphics[
natheight=3.550200in,
natwidth=3.651900in,
height=3.551in,
width=3.6519in
]%
{F:/KeplerNewtonHooke/Figures/Figure3.jpg}%
\caption{Illustration of the optical property of the ellipse\label{F3a}}%
\end{figure}

As usual, we shall always denote the constant $\overline{PF_{1}}%
+\overline{PF_{2}}$ of a given ellipse $\Gamma$ by $2a$, the distance between
$F_{1}$ and $F_{2}$ by $2c$, and $b=\sqrt{a^{2}-c^{2}}$. The sign // reads
\textquotedblleft is parallel to\textquotedblright.

\begin{corollary}
Let $d_{1}$(resp. $d_{2})$ be the distance between $F_{1}$(resp. $F_{2})$ and
a tangent line $\mathcal{T}_{P}$. Then
\begin{equation}
d_{1}d_{2}=b^{2}. \label{d1d2}%
\end{equation}

\end{corollary}

\begin{proof}
Let $\left\{  F_{1}^{\prime},F_{2}^{\prime}\right\}  $ be the reflection
points of $\left\{  F_{1},F_{2}\right\}  $ w.r.t. $\mathcal{T}_{P}$, see
Figure . Then $\overline{F_{1}F_{2}^{\prime}}$ and $\overline{F_{1}^{\prime
}F_{2}}$ have length $2a$ and intersect at $P$. By the Pythagorean Theorem,
applied to $\Delta F_{1}F_{2}^{\prime}H$ and $\Delta F_{1}^{\prime}%
F_{2}^{\prime}H$, one has%
\begin{align*}
4a^{2}  &  =(d_{1}+d_{2})^{2}+(\overline{F_{2}^{\prime}H})^{2}\\
4c^{2}  &  =(d_{1}-d_{2})^{2}+(\overline{F_{2}^{\prime}H})^{2}%
\end{align*}
and the identity (\ref{d1d2}) follows from this.
\end{proof}

\begin{corollary}
As indicated in Figure \ref{F3a}, if $K$ is the point on\ $\overline{PF_{1}}%
$\ so that $\overline{OK}$ $//\mathcal{T}_{P}$, then $\overline{PK}=a.$
\end{corollary}

\begin{proof}
Let $E$ be the point on $\overline{PF_{1}}$so that $\overline{PE}%
=\overline{PF_{2}}$. Then
\[
\overline{F_{2}E}\perp\mathcal{V}_{P}\text{ \ }\Longrightarrow\overline
{F_{2}E}\text{ }//\text{ }\mathcal{T}_{P}\text{ }//\overline{\text{ }OK}%
\]
and hence $\overline{F_{1}K}=\overline{EK}$ and%
\[
\ \text{\ }2a=\overline{PF_{1}}+\overline{PF_{2}}=(\overline{PK}+\overline
{EK})+\overline{PE}=2\overline{PK}.
\]

\end{proof}

\subsection{A remarkable formula for the curvature of the ellipse}

\begin{theorem}
\label{curvature}As illustrated in Figure \ref{F3a}, set $\varepsilon$ to be
the angle between $\overline{PF_{1}}$ and $\mathcal{T}_{P}$, and $\rho$ to be
the radius of the osculating circle of $\Gamma$ at $P$ (i.e. the radius of
curvature). Then
\begin{equation}
\kappa=\frac{1}{\rho}=\frac{a}{b^{2}}\sin^{3}\varepsilon\label{secret}%
\end{equation}

\end{theorem}

\begin{proof}
Let us begin with a pertinent fact on circular motions which led Hooke to
grasp the dynamical significance of curvature. A circular motion with radius
$\rho$ can be represented by
\[
\left\{
\begin{array}
[c]{c}%
x=\rho\cos\theta(t)\\
y=\rho\sin\theta(t)
\end{array}
\right.  \text{ \ or \ \ }\overrightarrow{OP}=\rho\binom{\cos\theta(t)}%
{\sin\theta(t)}%
\]
Thus, using Newton's notation of $\dot{\theta}=\frac{d\theta}{dt}$, etc.,
\begin{align*}
\overrightarrow{v}  &  =\binom{\dot{x}}{\dot{y}}=\rho\dot{\theta}\binom
{-\sin\theta}{\cos\theta}\text{, \ }|\overrightarrow{v}|^{2}=(\rho\dot{\theta
})^{2}\\
\overrightarrow{a}  &  =\binom{\ddot{x}}{\ddot{y}}=\rho\dot{\theta}^{2}%
\binom{-\cos\theta}{-\sin\theta}+\rho\ddot{\theta}\binom{-\sin\theta}%
{\cos\theta}\text{, \ \ }\overrightarrow{n}=\binom{-\cos\theta}{-\sin\theta}%
\end{align*}
Therefore%
\begin{align*}
\overrightarrow{a}\cdot\overrightarrow{n}  &  =\rho\dot{\theta}^{2}%
=\frac{|\overrightarrow{v}|^{2}}{\rho}\\
\kappa &  =\frac{1}{\rho}=\frac{\overrightarrow{a}\cdot\overrightarrow{n}%
}{|\overrightarrow{v}|^{2}}=\frac{|\overrightarrow{v}\times\overrightarrow{a}%
|}{|\overrightarrow{v}|^{3}}=\frac{\dot{x}\ddot{y}-\dot{y}\ddot{x}}{(\dot
{x}^{2}+\dot{y}^{2})^{3/2}}\
\end{align*}

Recall that the osculating circle at a point approximates a ($C^{2}$-smooth)
curve up to second order accuracy. Therefore, one can apply the above formula
for circular motions to the osculating circle at a point $P$. Thus, the
localization of the dynamics on such a curve at $P$ is essentially the same as
that of a corresponding motion on its osculating circle at $P$, and hence
\begin{equation}
\overrightarrow{a}_{P}\cdot\overrightarrow{n}_{P}=\frac{|\overrightarrow{v_{P}%
}|^{2}}{\rho}\text{, \ }\kappa=\frac{\dot{x}\ddot{y}-\dot{y}\ddot{x}}{(\dot
{x}^{2}+\dot{y}^{2})^{3/2}}\text{\ } \label{a-normal}%
\end{equation}
holds in general. This is the physical meaning of $\overrightarrow{a}_{P}%
\cdot\overrightarrow{n}_{P}$, the normal component of the acceleration.

Next, let us use the simple dynamical representation of a given ellipse
$\Gamma$, namely%
\[
x=a\cos t\text{, \ }y=b\sin t\text{,}
\]
to compute the curvature of $\Gamma$ at $P$, as follows:%

\[
\overrightarrow{v}=\binom{-a\sin t}{b\cos t}\text{, \ }\overrightarrow{a}%
=\binom{-a\cos t}{-b\sin t}=\overrightarrow{PO}\text{\ }%
\]
As can be seen from Figure \ref{F3a}, the area of the parallelogram spanned by
$\overrightarrow{v}$ and $\overrightarrow{a}$ is
\[
Area(//(\overrightarrow{v},\overrightarrow{a}))=|\overrightarrow{v}%
|(\overrightarrow{a}\cdot\overrightarrow{n})=ab
\]
and combined with (\ref{a-normal}) one has%

\begin{equation}
\frac{1}{\rho}|\overrightarrow{v}|^{2}=\overrightarrow{a}\cdot
\overrightarrow{n}\text{ }\Longrightarrow\frac{1}{\rho}=\frac{ab}%
{|\overrightarrow{v}|^{3}} \label{a.n}%
\end{equation}
On the other hand, by Corollary 3.4 and Figure \ref{F3a}, one also has
$\overline{PK}=a$, so $\overrightarrow{a}\cdot\overrightarrow{n}%
=a\sin\varepsilon$. Therefore, by (\ref{a.n})%
\[
ab=|\overrightarrow{v}|a\sin\varepsilon,\text{ \ i.e. }|\overrightarrow{v}%
|=\frac{b}{\sin\varepsilon},
\]
and formula (\ref{secret}) follows immediately from this.
\end{proof}

\subsection{The polar coordinate equation of an ellipse}

\begin{theorem}
Set $\overline{F_{1}P}=r$ and $\theta=\angle F_{2}F_{1}P$. Then the equation
of the ellipse $\Gamma$ is given by%
\begin{equation}
\frac{1}{r}=\frac{1}{b^{2}}(a-c\cos\theta) \label{ellipse}%
\end{equation}

\end{theorem}

\begin{proof}
$\overline{PF_{2}}=2a-r$, and by the cosine law applied to triangle $\Delta
F_{2}F_{1}P$%
\[
(2a-r)^{2}=r^{2}+4c^{2}-4cr\cos\theta,
\]
and consequently%
\[
4b^{2}=4a^{2}-4c^{2}=4r(a-c\cos\theta),
\]
which can be restated as in (\ref{ellipse}).
\end{proof}

\section{On the derivation of the inverse square law as a consequence of
Kepler's area law and ellipse law}

First of all, it follows readily from the area law, namely the quantity in
(\ref{2}) is a constant, that the acceleration vector $\overrightarrow{a}$ is
pointing towards $F_{1}$, see Theorem \ref{radial}. Thus, what remains to
prove is that the magnitude of $\overrightarrow{a}$ should be inverse
proportionate to the square of $\overline{PF_{1}}$ as a consequence of the
ellipse law. This is the monumental achievement of Newton which also led him
to the great discovery of the universal gravitation law, see Section 7.
However, his proof (cf. Proposition 11 of \cite{Newton}) is rather difficult
to understand.

The following are three much simpler proofs, each uses different aspects of
the focal geometry of ellipses discussed in Section 3:

\ \ \ \ \ \ \ \ \ \ \ \ \ \ \ \ \ \ \ \ \ \ \ \ \ \ \ \ \ \ \ \ \ \ \ \ \ \ \ \ \ \ 

\begin{proof}
I (in the spirit of Greek geometry)

By the area law and the identity (\ref{d1d2})%
\[
d_{1}|\overrightarrow{v}|=\frac{2\pi ab}{T},\text{ \ \ }d_{1}d_{2}=b^{2}%
\]
Therefore,
\begin{align*}
|\overrightarrow{v}|  &  =\frac{2\pi ab}{Td_{1}}=\frac{2\pi ab}{T}\frac{d_{2}%
}{b^{2}}=\frac{\pi a}{bT}\overline{F_{2}F_{2}^{\prime}}\\
\overrightarrow{F_{2}F_{2}^{\prime}}  &  =\overrightarrow{F_{2}F_{1}%
}+\overrightarrow{F_{1}F_{2}^{\prime}}=\binom{-2c}{0}+2a\binom{\cos\theta
}{\sin\theta}%
\end{align*}
Hence (see Figure \ref{F3a}),
\[
\overrightarrow{v}=\frac{\pi a}{bT}\binom{0}{-2c}+\frac{2\pi a^{2}}{bT}%
\binom{-\sin\theta}{\cos\theta}%
\]
and by Remark \ref{area law}
\[
\overrightarrow{a}=\frac{d}{dt}\overrightarrow{v}=\frac{2\pi a^{2}\dot{\theta
}}{bT}\binom{-\cos\theta}{-\sin\theta}=-\frac{\pi^{2}}{2}\frac{(2a)^{3}}%
{T^{2}}\frac{1}{r^{2}}\binom{\cos\theta}{\sin\theta}%
\]

\ \ \ \ \ \ \ \ \ \ \ \ \ \ \ \ \ \ \ \ \ \ \ \ \ \ \ \ 
\end{proof}

\begin{proof}
\ II (using the kinematic formula for curvature--- the proof Hooke sorely
missed ?)

By the area law (Remark \ref{area law}),Theorem \ref{curvature}, and
(\ref{a-normal}),%
\begin{align*}
r|\overrightarrow{v}|\sin\varepsilon &  =\frac{2\pi ab}{T}\\
|\overrightarrow{a}|\sin\varepsilon &  =\overrightarrow{a}\cdot
\overrightarrow{n}=\frac{1}{\rho}|\overrightarrow{v}|^{2}\text{, \ }\frac
{1}{\rho\sin^{3}\varepsilon}=\frac{a}{b^{2}},
\end{align*}
and consequently%
\begin{equation}
|\overrightarrow{a}|=\frac{1}{\rho\sin\varepsilon}|\overrightarrow{v}%
|^{2}=\frac{4\pi^{2}a^{3}}{T^{2}}\frac{1}{r^{2}} \label{|a|}%
\end{equation}

\begin{remark}
On page 110 of \cite{Chand}, S. Chandrasekhar states: \textquotedblleft That
Newton must have known this relation (cf. (\ref{secret})) requires no
argument!\textquotedblright. In fact, a thorough analysis of the proof of
Proposition 11 in \cite{Newton} will reveal that its major portion is devoted
to the proof of (\ref{secret}), and the inverse square law can then be deduced
essentially in the same way as the short simple step of Proof II. But such a
crucial role of curvature in his proof is, somehow, hidden in his presentation.
\end{remark}
\end{proof}

\ \ \ \ \ \ \ \ \ \ \ \ 

\begin{proof}
III (using analytic geometry)

For a planary motion with position vector
\begin{equation}
\overrightarrow{OP}=r\left(
\begin{array}
[c]{c}%
\cos\theta\\
\sin\theta
\end{array}
\right)  \label{position}%
\end{equation}
the acceleration vector is
\begin{equation}
\overrightarrow{a}=\frac{d^{2}}{dt^{2}}\overrightarrow{OP}=(\ddot{r}%
-r\dot{\theta}^{2})\left(
\begin{array}
[c]{c}%
\cos\theta\\
\sin\theta
\end{array}
\right)  +(2\dot{r}\dot{\theta}+r\ddot{\theta})\left(
\begin{array}
[c]{c}%
-\sin\theta\\
\cos\theta
\end{array}
\right)  \label{a-decomp}%
\end{equation}
Now, assuming Kepler's area law the force must be radial, and therefore the
second component in (\ref{a-decomp}) vanishes. Moreover, assuming the
trajectory is an ellipse (or conic section, but not a circle), we need only
show that $(\ddot{r}-r\dot{\theta}^{2})$ is inverse proportional to $r^{2}$.
This will be achieved by differentiation of the polar coordinate equation
(\ref{ellipse}).

At this point, however, it is instructive to derive the Binet equation and
apply it to our situation, since the remaining calculations will be similar in
both cases. Thus, setting $q=1/r$ as a new variable depending on $\theta$,
straightforward differentiation of $q$ and elimination of $\dot{\theta}$ by
introducing the constant $k=r^{2}\dot{\theta}$ yield the identity
\begin{equation}
\ddot{r}-r\dot{\theta}^{2}=-k^{2}q^{2}\left(  \frac{d^{2}q}{dq^{2}}+q\right)
\text{,} \label{Binet2}%
\end{equation}
which is just the Binet equation (\ref{Binet}) divided by $m$. Next, by
differentiation of the ellipse equation (\ref{ellipse})
\[
q=\frac{1}{b^{2}}(a-c\cos\theta),
\]
we deduce the identity
\[
\frac{d^{2}q}{dq^{2}}+q=\frac{a}{b^{2}}%
\]
which by substitution into (\ref{Binet2}) yields%
\begin{equation}
\ddot{r}-r\dot{\theta}^{2}=-\frac{k^{2}a}{b^{2}}\frac{1}{r^{2}}=-\frac
{4\pi^{2}a^{3}}{T^{2}}\frac{1}{r^{2}}. \label{eq3}%
\end{equation}

\end{proof}

\begin{remark}
The proof given by Newton in \cite{Newton} is historically the \textit{first
proof }of the inverse square law, a great historical event among the major
advances of the civilization of rational mind. Therefore, a careful reading as
well as a thorough understanding of the underlying pertinent ideas (or
insights) of such a proof are, of course, highly desirable and of great
significance towards one's understanding of the history of science. For
example, a reading of corresponding sections of \cite{Chand} might be helpful
for such an undertaking.
\end{remark}

\begin{remark}
If one compares Newton's proof and the above triple of proofs, one finds that
the \textit{area law and the focal geometry of ellipses always play the major
roles in each proof, while the differentiation of sine and cosine is the only
needed analytical computation involved in each proof. In fact, the main
differences between them lie in the ways of proper synthesis between the area
law and the focal geometry of ellipses.}
\end{remark}

\section{A crucial integration formula for the gravitation attraction of a
spherically symmetric body}

Newton's proof (cf. Proposition 71 in Book I of \cite{Newton}) for the
following important integration formula, nowadays often referred to as the
\textquotedblleft superb theorem\textquotedblright, is again not easy to
understand. Therefore, for the convenience of the reader, we include here an
elementary simple proof (cf. \cite{Hsiang}). In the very recent physics
literature, see Schmid \cite{Schmid} for a similar but still different
geometric proof.

\begin{theorem}
The total gravitation force acting on an outside particle by a body with
spherically symmetric mass distribution is equal to that of a point mass of
its total mass situated at its center.
\end{theorem}

\begin{proof}
First of all, the proof can be directly reduced to the case of a thin
spherical shell with a uniform mass density $\rho$ per area. The key idea of
this elementary geometric proof is to use the subdivision of the spherical
surface of radius $R$ \ induced by the subdivision of the total solid angle
(i.e. the unit sphere) centered at $P^{\prime}$ on $\overline{OP}$ with
$\overline{OP^{\prime}}\cdot\overline{OP}=R^{2}$, as illustrated in Figure 4.%

\begin{figure}[ptb]%
\centering
\includegraphics[
natheight=3.610700in,
natwidth=5.948500in,
height=3.6107in,
width=5.931in
]%
{F:/KeplerNewtonHooke/Figures/Figure4.jpg}%
\caption{Newton's \textquotedblleft superb theorem\textquotedblright}%
\end{figure}

For any given point $Q$ on the $R$-sphere, $\Delta OPQ$ and $\Delta
OQP^{\prime}$ have the same angle at $O$, and moreover, their corresponding
pairs of sides are in proportion, namely $P`P^{\prime}$ $\overline{P^{\prime}%
}P\prime$%
\[
\frac{\overline{OP}}{\overline{OQ}}=\frac{\overline{OP}}{R}=\frac
{R}{\ \overline{OP^{\prime}}}=\frac{\overline{OQ}}{\overline{OP^{\prime}}}.
\]
Therefore, $\Delta OPQ$ $\sim\Delta OQP^{\prime}$ and hence%
\[
\angle OQP^{\prime}=\angle OPQ\text{ \ }(:=\theta)\text{ \ and \ }%
\frac{\overline{P^{\prime}Q}}{\overline{PQ}}=\frac{R}{\overline{OP}}.
\]

Now, let $dA$ be the element of area on the $R$-sphere around $Q$ and
$d\sigma$ be its corresponding element of area on the unit sphere. Then, as
indicated in the magnified solid angle cone of $d\sigma$ in Figure 4, the
corresponding area element on the $\overline{P^{\prime}Q}$ -sphere centered at
$P^{\prime}$ is equal to $\overline{P^{\prime}Q}$ $^{2}d\sigma$ on the one
hand, but equals $dA\cos\theta$ on the other hand, because the dihedral angle
between the tangent planes of $dA$ (resp. $\overline{P^{\prime}Q}$
$^{2}d\sigma)$ at $Q$ is equal to $\theta$. Consequently,
\[
dA\cos\theta=\overline{P^{\prime}Q}^{2}d\sigma.
\]
\qquad\ Note that the contribution of $d\overrightarrow{F}$ to the total
composite force is equal to $|d\overrightarrow{F}|\cos\theta$, namely (with
mass $m_{1}$ at $P$)%
\begin{align*}
|d\overrightarrow{F}|\cos\theta &  =G\frac{m_{1}\rho dA}{|\overline{PQ}|^{2}%
}\cos\theta=Gm_{1}\rho\frac{\overline{P^{\prime}Q}^{2}}{\overline{PQ}^{2}%
}d\sigma=Gm_{1}\rho\frac{R^{2}}{\overline{OP}^{2}}d\sigma.\\
&  \
\end{align*}
Therefore, the total gravitation force is given by%
\begin{align}%
{\displaystyle\int}
|d\overrightarrow{F}|\cos\theta &  =Gm_{1}\rho\frac{R^{2}}{\overline{OP}^{2}}%
{\displaystyle\int}
d\sigma=G\frac{m_{1}4\pi R^{2}\rho}{\overline{OP}^{2}}\label{grav}\\
&  =G\frac{m_{1}m_{2}}{\overline{OP}^{2}},\text{ \ }m_{2}=4\pi R^{2}%
\rho.\nonumber
\end{align}

\begin{remark}
(i) Newton was undoubtedly aware of the importance of the integration formula
for the gravitational force of a spherically symmetric body, both for
celestial gravity and for terrestrial gravity, and moreover, for the
unification of both, thus enabling him to proclaim the law of universal gravitation.

(ii) In fact, he must have been working hard on it, ever since his success in
proving the inverse square law some time in 1680--81. His letter of June 20,
1686, to Halley recorded his repeated failure up to around 1685, while his
final success, in 1686, of proving such a wonderful simple formula is actually
the crown-jewel of his glorious triumph --- the law of universal gravitation
(cf. \cite{Chand}, \cite{Newton}).
\end{remark}
\end{proof}

\section{A simple proof of the Kepler problem}

For the convenience of the reader, we include here a simple proof on the
solution of the Kepler problem, as follows:

\begin{theorem}
\label{Kepler problem}Suppose that the acceleration of a motion is centripetal
and inverse proportional to the square of $\overline{OP}=r$, namely for some
constant $C>0$,%
\begin{equation}
\overrightarrow{a}=\frac{C}{r^{2}}\binom{-\cos\theta}{-\sin\theta}.
\label{accel}%
\end{equation}
Then the motion satisfies the area law and its orbit is a conic section with
the center as one of its foci.
\end{theorem}

\begin{proof}
I (vector calculus involving only $\sin x$ and $\cos x$)

The centripetality property amounts to the area law (see Section
\ref{area law}), so there exists a constant $k$ such that%

\begin{equation}
\text{ }2\text{$\frac{dA}{dt}$}=\text{\ }r^{2}\dot{\theta}=k \label{eq4}%
\end{equation}
Therefore, it follows directly from (\ref{accel}) and (\ref{eq4}) that%
\[
\frac{d}{d\theta}\overrightarrow{v}(\theta)=\overrightarrow{a}(\theta
)\frac{dt}{d\theta}=\frac{C}{k}\frac{d}{d\theta}\binom{-\sin\theta}{\cos
\theta}%
\]
Hence, there exists a constant vector $\overrightarrow{\delta}$ such that%
\[
\overrightarrow{v}(\theta)=\frac{C}{k}\binom{-\sin\theta}{\cos\theta
}+\overrightarrow{\delta}.
\]
Without loss of generality, we may assume that $\overrightarrow{v}(0)$ is
pointing vertically, thus having $\overrightarrow{\delta}=\binom{0}{\delta}$.

Now, again using the area law, one has from (\ref{ang}) and (\ref{eq4})%
\[
r\left\vert
\begin{array}
[c]{cc}%
\cos\theta & \frac{C}{k}(-\sin\theta)\\
\sin\theta & \frac{C}{k}\cos\theta+\delta
\end{array}
\right\vert =k,
\]
namely%
\begin{equation}
\frac{1}{r}=\frac{C}{k^{2}}(1+e\cos\theta),\text{ \ }e=\frac{k\delta}{C},
\label{conic3}%
\end{equation}
which is exactly the polar coordinate equation (\ref{ellipse}) of a conic
section with $\pm e$ as its eccentricity and $O$ as a focal point.
\end{proof}

\ \ \ \ \ \ \ \ \ \ \ \ \ \ \ \ \ \ \ \ \ \ \ \ \ \ \ \ \ \ \ \ \ \ \ \ \ \ 

\begin{proof}
II (integration, starting from the Binet equation)

By assumption, the radial force is of type $F=-\alpha q^{2},\alpha>0$
constant, and $q=1/r$. Therefore, by the Binet equation (\ref{Binet})%
\[
\frac{d^{2}q}{d\theta^{2}}+q=h>0\text{ (constant)}%
\]
Since the general solution of this 2nd order ODE can be written as
$q=A\cos(\theta+\theta_{0})+h$, with the appropriate choice of the axis
$\theta=0$ the solution takes the form (\ref{conic3}).
\end{proof}

\begin{remark}
(i) For the Kepler problem, the constant $C$ in (\ref{accel}) is positive
(i.e. in the case of attraction force). However, the proof also works as well
for $C<0$ (i.e. repulsive force), which is important in quantum mechanics for
studying scattering.

(ii) Mathematically, the solution of the Kepler problem amounts to solve the
second order ODE (\ref{accel}) in terms of the given initial data, especially
the uniqueness without \ appealing to sophisticated theorems. Our first proof
of Theorem \ref{Kepler problem} accomplishes such a task in two simple,
elementary steps, namely, firstly obtaining the solution of
$\overrightarrow{v}(\theta)$ in terms of the initial velocity by a direct
application of the area law (i.e. consequence of the centripetality), and then
obtaining the polar coordinate equation of the trajectory by another direct
application of the area law.

Note that the area law is actually the dynamical manifestation of the
rotational symmetry of the plane with respect to the center of the
centripetality. Therefore, it is, of course, natural to use the polar
coordinate system and compute $\frac{d}{d\theta}\overrightarrow{v}(\theta)$ in
the proof. In retrospect, it is not only the simplest proof with perfect
generality and the least of technicality, but it is also the most natural way
of solving the Kepler problem.

(iii) We refer to \cite{Arnold}, \cite{Chand}, \cite{Newton}, \cite{Speis} for
comparison of proofs of Theorem \ref{Kepler problem}, as well as for the
discussions of whether Newton actually proved it.
\end{remark}

\section{Concluding remarks}

\qquad\textbf{(i)} In Book III of \emph{Principia}, Newton presents his
crowning achievements, namely a demonstration of the structure of the
\textquotedblleft system of the world\textquotedblright, derived from the
basic principles that he had developed in Book I and Book II. Newton's three
laws of motion and the law of universal gravitation are for the first time
seen to provide a unified quantitative explanation for a wide range of
physical phenomena. In particular, they provide the foundation of celestial
mechanics, and the first complete mathematical formulation of the classical
$n$-body problem appears in Newton's Principia. The law of universal
gravitation is, in fact, the first and also one of the most important
scientific discoveries in the entire history of sciences. An in-depth
understanding of how it arises naturally from the mathematical analysis as
well as synthesis of those empirical laws of Kepler and Galilei is not only
instructional but also inspirational.

\ \ \ \ \ \ \ \ 

\textbf{(ii)} The law of universal gravitation reflects the physical principle
expounded by Newton that all bodies interact gravitationally. But such a
statement presupposes a deeper understanding of the force law $F=ma$, namely
that two interacting bodies attract each other by equal forces and in opposite
directions. This follows from \textit{Newton's third law}, which is his own
insight, stating that for every action there is an opposite reaction. In the
case of gravitation this interaction is expressed by the basic and well known
formula
\begin{equation}
F=G\frac{mM}{r^{2}} \label{gravity}%
\end{equation}
for the mutual gravitation force between two point masses $m$ and $M$
separated by the distance $r$, where $G$ is the \textit{gravitational
constant. }By Newton's \textquotedblleft superb theorem\textquotedblright%
\ (see Section 5), the same formula holds for two bodies with spherically
symmetric mass distributions and total masses $m$ and $M$, and $r$ is the
distance between their centers. For many bodies, such as the planets circling
the sun, the bodies attract one another and therefore they also perturb one
another's orbits. Still, as pointed out by Newton, the law of universal
gravitation explains why the planets follow Kepler's laws approximately and
why they depart (as is also observed) from the laws in the way they do. Let us
briefly recall the underlying reasoning.

First, consider a sun-planet system with masses $M$ and $m$, ignoring the
other celestial bodies. By combining the force law $\overrightarrow{F}%
=m\overrightarrow{a}$ and formula (\ref{gravity}), it may seem that one is led
to equation (\ref{accel}), with $C=GM$, and thus the planet's orbit will be a
solution of the Kepler problem as stated in Theorem \ref{Kepler problem},
namely an ellipse with the sun at one of the focal points. However, this
reduction of the sun-planet problem to a one-body (or Kepler) problem centered
at the sun is only approximately correct. As we would phrase it today, the
validity of the Newtonian dynamics hinges upon using an inertial frame of
reference, namely with the origin \textquotedblleft at rest\textquotedblright.

How did Newton himself imagine the origin of an inertial frame could be
chosen? The sun is much larger than the planets, but Newton was aware of the
tiny motion of the sun due to the attraction of the planets. He estimated the
Center of the World, namely the center of gravity of the whole solar system,
to be very close to the sun, say within one solar diameter.

For a general two-body system, the common center of gravity is
\textquotedblleft at rest\textquotedblright\ if the interaction with other
bodies is neglected. Thus, the position of one body determines the position of
the other, and Newton argues correctly that the two-body problem again reduces
to a one-body problem with radial attraction towards the center of gravity.
So, both bodies follow Keplerian orbits with the latter point as a common focus.

However, whereas an exact solution of the two-body problem is one of the great
triumphs of classical mechanics, the non-integrability of the $n$-body problem
for $n$ $\geq3$, which is well known nowadays, was maybe suspected already by
Newton when he wrote in his tract \textit{De Motu (1684): }\textquotedblleft%
--- to define these motions by exact laws allowing of convenient calculation
exceeds, unless I am mistaken, the force of the entire human
intellect\textquotedblright.

\ \ \ \ \ \ \ \ 

\textbf{(iii)} The measurement of the gravitational constant $G$ in formula
(\ref{gravity}) has a long history; in fact, the formulation of gravity in
terms of $G$ did not become standard until the late 19th century. The first
successful experiment in the laboratory, by H. Cavendish (1731--1810) in 1798,
aimed at measuring the mass $M_{e}$ of the earth, or equivalently, the
(average) density $\rho_{e}$ of the earth from the knowledge of the earth's
radius $R$. However, knowing the acceleration of gravity $g$ at the surface of
the earth, measuring $\rho_{e}$ amounts to measuring $G$ due to the relations%

\[
G=\frac{gR^{2}}{M_{e}}=\frac{3g}{4\pi R\rho_{e}}.
\]

The apparatus used by Cavendish was actually designed by the geologist J.
Michell (1724--1793), who was a pioneer in seismology and did also important
work in astronomy. Although Laplace (1796) is usually credited for being the
first who described the concept of a black hole (condensed star), Michell
argued in a 1784 paper how such an objects could be observed from its
gravitational effect on nearby objects. However, Michell is best known for his
invention, probably in the early 1780's, of the torsion balance, which is the
major device of the apparatus he built to measure the quantity $\rho_{e}$.

But Michell did not complete this project, and his equipment was taken over by
Cavendish, who rebuilt the apparatus with some improvements, which enabled him
to carry out measurements of the density of the earth with very high accuracy.
We refer to his report \cite{Cavend}, see also \cite{Shamos}, \cite{Clot}. The
measurement of the universal constant $G$ had many remarkable consequences,
for example, estimates of the mass of the earth, moon, sun, other planets, and
massive black holes.

\ \ \ \ 

\textbf{(iv) }In 1785 Coulomb published his investigation of the electric
force, using an apparatus involving a torsion balance. But, according to
Cavendish, Michell had described his torsion balance device to him before
1785, so it seems that both Michell and Coulomb must be credited with the
invention of the ingenious torsion balance. Like the gravitation force, the
Coulomb force between charged particles is also of the inverse square type. In
fact, here Newton's \textquotedblleft superb theorem\textquotedblright\ not
only applies, but also plays a useful role in providing the empirical evidence
as well as the measurement of the constant of proportionality, namely the
\textit{electric force constant }(or Coulomb's constant)\textit{ }$k_{e}$.

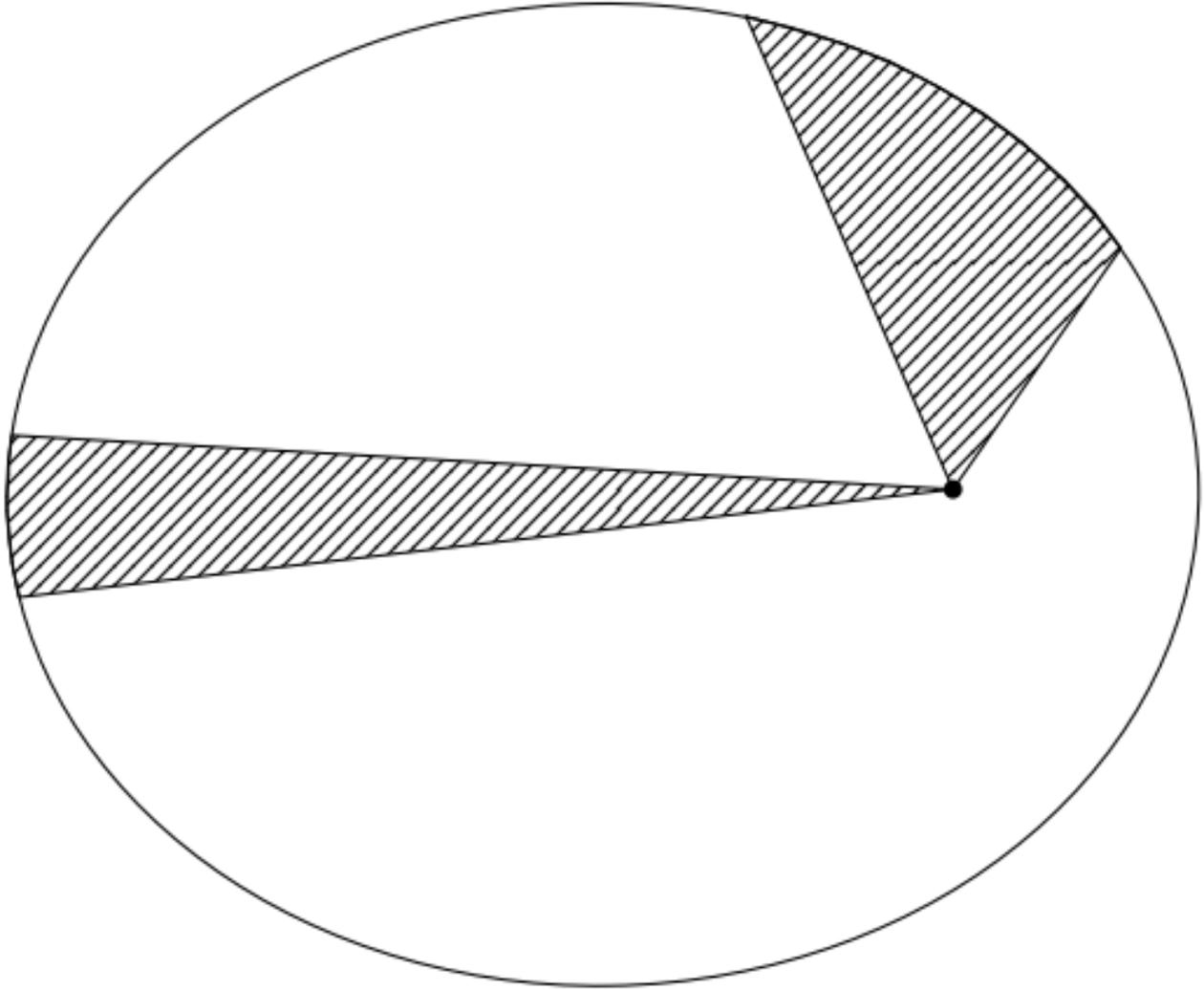
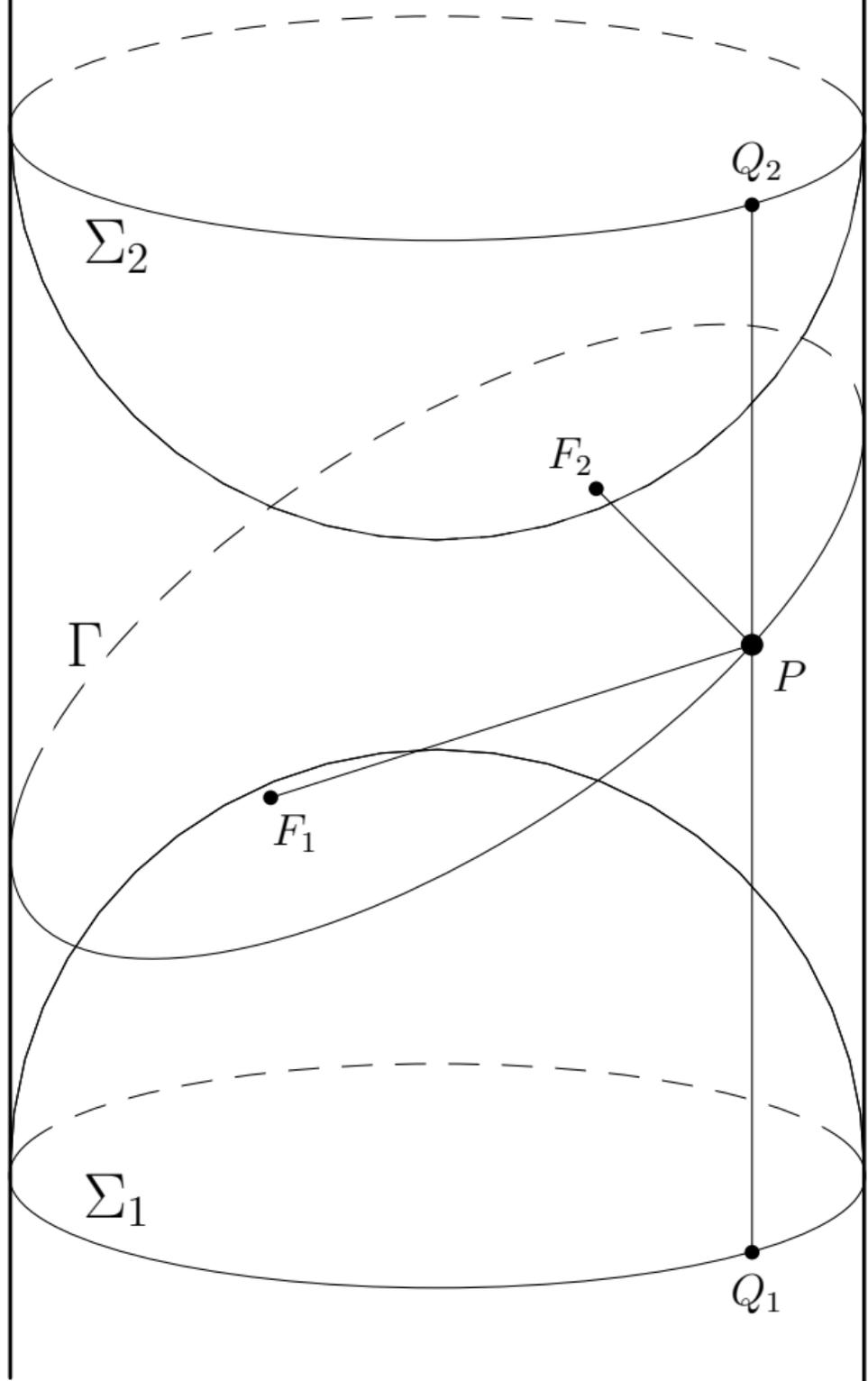
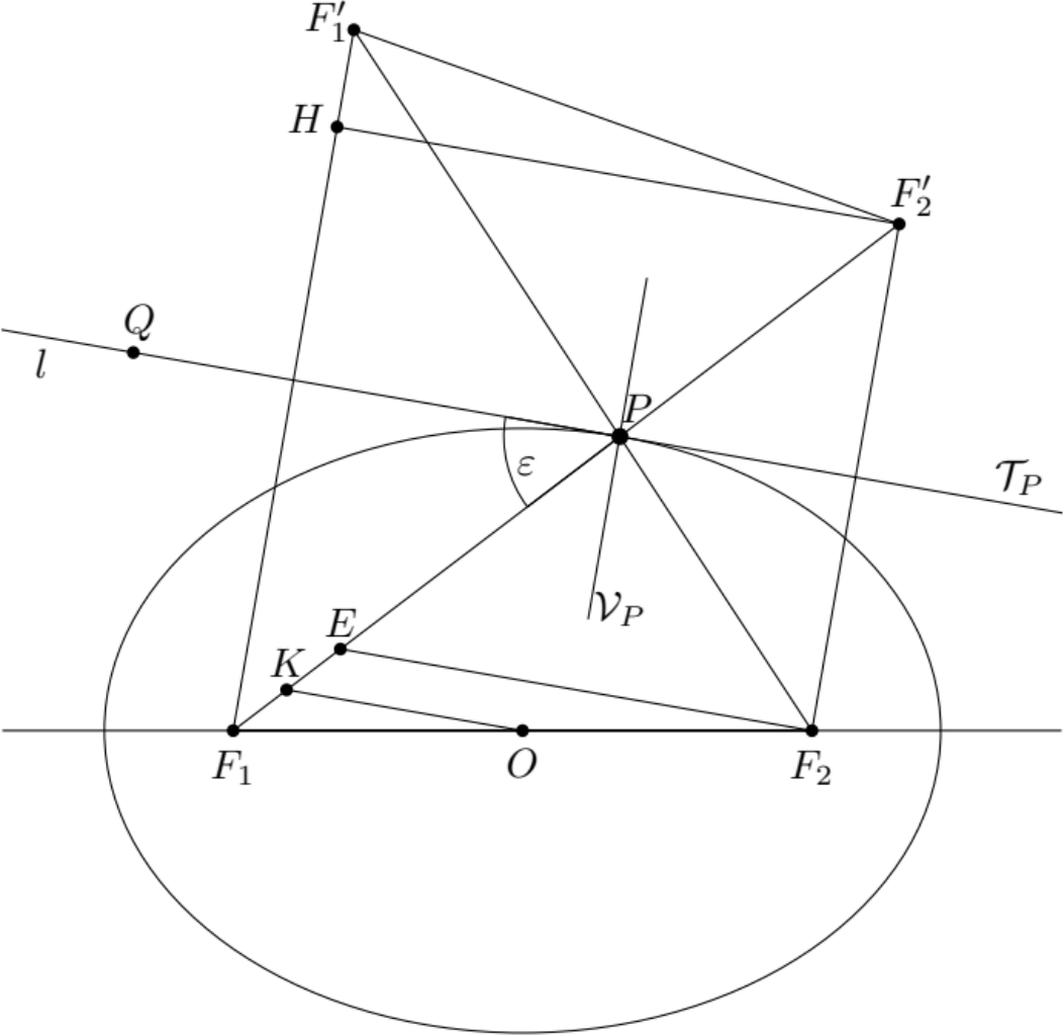
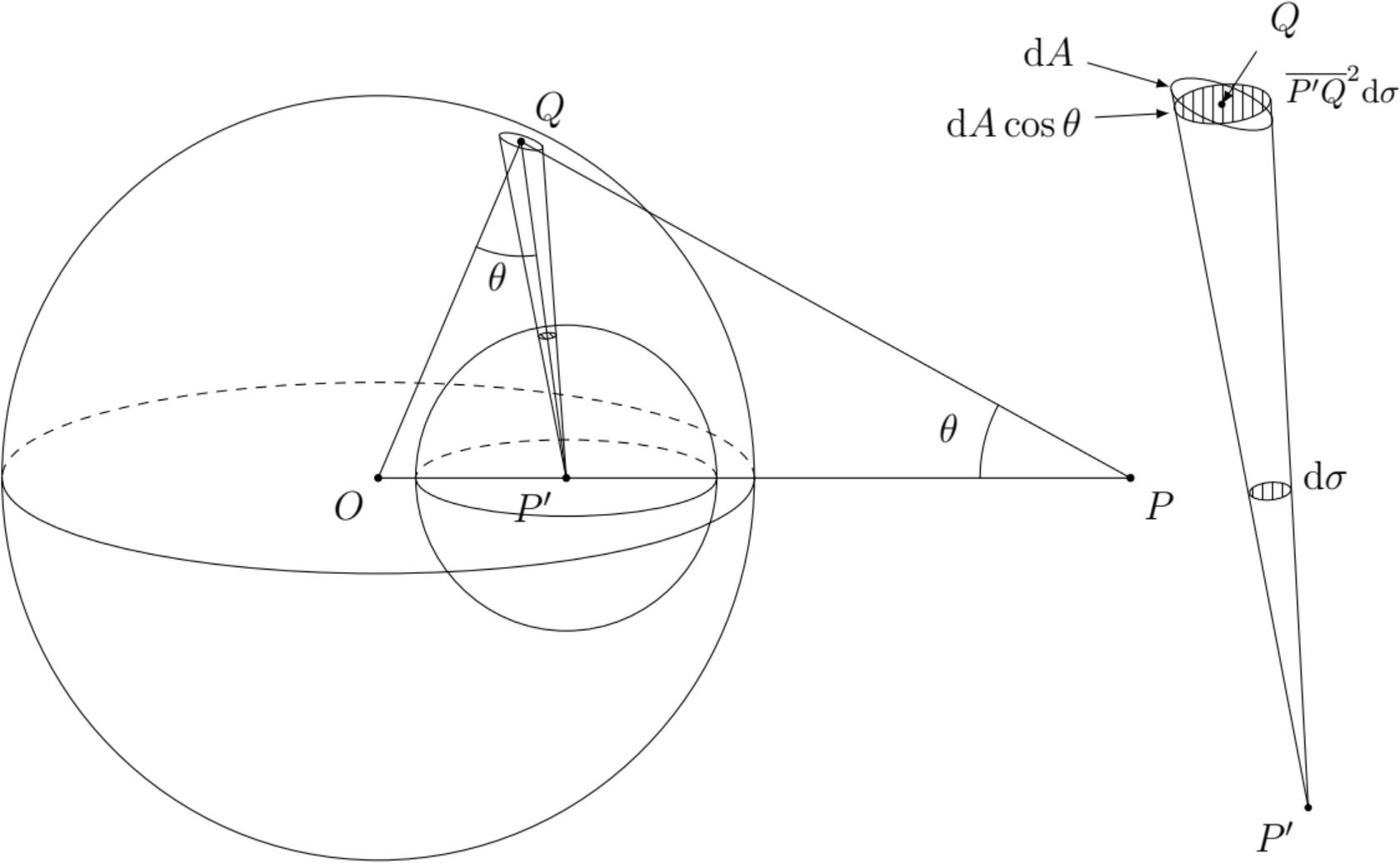
\end{document}